\documentclass[11pt,letterpaper]{amsart}

\usepackage{amsmath}
\usepackage{amsthm}
\usepackage{amssymb}

\newtheorem{theorem}{Theorem} 
\newtheorem{lemma}[theorem]{Lemma}

\newcommand{\z}{\zeta}

\newcommand{\C}{\mathbb{C}}
\newcommand{\R}{\mathbb{R}}
\newcommand{\norm}[1]{\lVert#1\rVert}
\newcommand{\abs}[1]{\lvert#1\rvert}
\newcommand{\bigabs}[1]{\big\lvert#1\big\rvert}

\newcommand{\biggabs}[1]{\bigg\lvert#1\bigg\rvert}
\newcommand{\Biggabs}[1]{\Bigg\lvert#1\Bigg\rvert}

\begin{document}

\title[Polynomials with unimodular coefficients]{On a problem due to Littlewood concerning polynomials with unimodular coefficients}

\author{Kai-Uwe Schmidt}

\date{13 September 2012 (revised 12 February 2013)}


\address{Faculty of Mathematics, Otto-von-Guericke University, Universit\"atsplatz~2, 39106 Magdeburg, Germany.}

\email{{\tt kaiuwe.schmidt@ovgu.de}}

\begin{abstract}
Littlewood raised the question of how slowly $\norm{f_n}_4^4-\norm{f_n}_2^4$ (where $\norm{.}_r$ denotes the $L^r$ norm on the unit circle) can grow for a sequence of polynomials $f_n$ with unimodular coefficients and increasing degree. The results of this paper are the following. For
\[
g_n(z)=\sum_{k=0}^{n-1}e^{\pi ik^2/n}\, z^k
\]
the limit of $(\norm{g_n}_4^4-\norm{g_n}_2^4)/\norm{g_n}_2^3$ is $2/\pi$, which resolves a mystery due to Littlewood. This is however not the best answer to Littlewood's question: for the polynomials
\[
h_n(z)=\sum_{j=0}^{n-1}\sum_{k=0}^{n-1} e^{2\pi ijk/n}\, z^{nj+k}
\]
the limit of $(\norm{h_n}_4^4-\norm{h_n}_2^4)/\norm{h_n}_2^3$ is shown to be~$4/\pi^2$. No sequence of polynomials with unimodular coefficients is known that gives a better answer to Littlewood's question. It is an open question as to whether such a sequence of polynomials exists.
\end{abstract}

\maketitle


\section{Introduction}

For real $r\ge 1$, the $L^r$ norm of a polynomial $f\in\C[z]$ on the unit circle is
\[
\norm{f}_r=\bigg(\frac{1}{2\pi}\int_0^{2\pi}\bigabs{f(e^{i\theta})}^r\,d\theta\bigg)^{1/r}.
\]
There is sustained interest in the $L^r$ norm of polynomials with restricted coefficients (see, for example, Littlewood~\cite{Littlewood1968}, Borwein~\cite{Borwein2002}, and Erd{\'e}lyi~\cite{Erdelyi2002} for surveys on selected problems). Littlewood raised the question of how slowly $\norm{f_n}_4^4-\norm{f_n}_2^4$ can grow for a sequence of polynomials $f_n$ with restricted coefficients and increasing degree. This problem is also of interest in the theory of communications, because $\norm{f}_4^4$ equals the sum of squares of the aperiodic autocorrelations of the sequence formed from the coefficients of $f$~\cite[p.~122]{Borwein2002}; in this context one considers the \emph{merit factor} $\norm{f}_2^4/(\norm{f}_4^4-\norm{f}_2^4)$. Much work on Littlewood's question has been done when the coefficients are $-1$ or $1$; see~\cite{Jedwab2011} for recent advances. In the situation where the coefficients are restricted to have unit magnitude, the polynomials 
\[
g_n(z)=\sum_{k=0}^{n-1}e^{\pi ik^2/n}\, z^k\quad\text{for integral $n\ge 1$}
\]
are of particular interest~\cite{Littlewood1961},~\cite{Littlewood1962},~\cite{Littlewood1966},~\cite{Littlewood1968}.\footnote{Some authors consider $g_n(e^{\pm\pi i/n}z)$, which however has the same $L^r$ norm as $g_n(z)$.} These polynomials are also the main ingredient in Kahane's celebrated semi-probabilistic construction of ultra-flat polynomials~\cite{Kahane1980}, which disproves a conjecture due to Erd\H{o}s~\cite{Erdos1962}. Write
\[
\alpha_n=\frac{\norm{g_n}_4^4-\norm{g_n}_2^4}{\norm{g_n}_2^3}
\]
(note that $\norm{f}_2=\sqrt{n}$ for every polynomial $f$ of degree $n-1$ with unimodular coefficients). Based on the work in~\cite{Littlewood1961} and~\cite{Littlewood1962} and calculations carried out by Swinnerton-Dyer, Littlewood concluded in~\cite{Littlewood1966} that 
\begin{equation}
\lim_{n\to\infty}\alpha_n=\sqrt{2}-\frac{2}{\pi}(\sqrt{2}-1)=1.15051\dots,   \label{eqn:littlewood_constant}
\end{equation}
but expressed doubt in his own conclusion. He knew that
\begin{equation}
0.604\le\alpha_n\le0.656\quad\text{for $18\le n\le 41$}   \label{eqn:numerical_data}
\end{equation}
and noted~\cite[Appendix]{Littlewood1966} ``There is a considerable mystery here. I have checked my calculations at least six times, and they have been checked also in great detail by Dr.~Flett.'' Littlewood raised this issue again in his book~\cite[p.~27]{Littlewood1968} and asked for a resolution of this puzzle.
\par
Borwein and Choi~\cite{Borwein2000} conjectured 
\[
\norm{g_n}_4^4=n^2+\frac{2}{\pi}n^{3/2}+\delta_nn^{1/2}+O(n^{-1/2}),
\]
where $\delta_n=-2$ for $n\equiv 0,1\pmod 4$ and $\delta_n=1$ for $n\equiv 2,3\pmod 4$ (this was not stated explicitly as a conjecture in~\cite{Borwein2000}, but was confirmed by the authors~\cite{Borwein2012} to be a tentative conclusion based on numerical evidence). This conjecture implies in particular
\begin{equation}
\lim_{n\to\infty}\alpha_n=\frac{2}{\pi}=0.63661\dots.   \label{eqn:limit_alpaha_conjectured}
\end{equation}
\par
Independently, Antweiler and B\"omer~\cite{Antweiler1990} made observations similar to~\eqref{eqn:numerical_data}, while Sta\'nczak and Boche~\cite{Stanczak2000} and Mercer~\cite{Mercer2012} derived bounds for $\alpha_n$. 
In particular, Mercer~\cite{Mercer2012} showed that
\[
\limsup_{n\to\infty}\,\alpha_n<\frac{16}{3\pi^{3/2}}=0.95779\dots,
\]
and thereby confirming Littlewood's suspicion (although Mercer was apparently unaware of Littlewood's work). 
\par
We shall resolve Littlewood's puzzle by proving that~\eqref{eqn:littlewood_constant} is incorrect and the conjecture~\eqref{eqn:limit_alpaha_conjectured} is true.
\begin{theorem}
\label{thm:main1}
We have
\[
\lim_{n\to\infty}\frac{\norm{g_n}_4^4-\norm{g_n}_2^4}{\norm{g_n}_2^3}=\frac{2}{\pi}.
\]
\end{theorem}
\par
We shall also show that this is not the best possible answer to Littlewood's question. To do so, we consider the polynomials
\[
h_n(z)=\sum_{j=0}^{n-1}\sum_{k=0}^{n-1}e^{2\pi ijk/n}\,z^{nj+k}\quad\text{for integral $n\ge 1$}
\]
of degree $n^2-1$, which have been studied by Turyn~\cite{Turyn1967}, among others. 
\begin{theorem}
\label{thm:main2}
We have
\[
\lim_{n\to\infty}\frac{\norm{h_n}_4^4-\norm{h_n}_2^4}{\norm{h_n}_2^3}=\frac{4}{\pi^2}.
\]
\end{theorem}
This is the best known answer to Littlewood's question: there is no sequence of polynomials $f_n$ with unimodular coefficients for which the limit of $(\norm{f_n}_4^4-\norm{f_n}_2^4)/\norm{f_n}_2^3$ is known to be less than~$4/\pi^2$. It is an open question as to whether such a sequence of polynomials exists.
\par
In the radar literature~\cite[Ch.~6]{Levanon2004}, the sequences formed from the coefficients of $g_n$ and $h_n$ are called \emph{Chu} and \emph{Frank sequences}, respectively. Our results show that their merit factors grow like $(\pi/2)\sqrt{n}$ and $(\pi^2/4)\sqrt{n}$, respectively, which explains numerical results reported in~\cite{Antweiler1990}.


\section{Proof of Theorem~\ref{thm:main1}}

We begin with summarising known results (see \cite[p.~371]{Littlewood1966}, for example). For a polynomial $f\in\C[z]$ with $f(z)=\sum_{k=0}^{d-1}a_kz^k$, we readily verify that
\[
f(z)\overline{f(z^{-1})}=\sum_{u=-(d-1)}^{d-1}c_uz^u,
\]
where
\begin{equation}
c_u=\sum_{0\le j,j+u<d} a_j\overline{a_{j+u}}.   \label{eqn:f_ck}
\end{equation}
The numbers $c_u$ satisfy $c_u=\overline{c_{-u}}$. Hence
\begin{equation}
\norm{f}_4^4=\frac{1}{2\pi}\int_0^{2\pi}\Big(f(e^{i\theta})\overline{f(e^{i\theta})}\Big)^2\, d\theta=c_0^2+2\sum_{u=1}^{d-1}\abs{c_u}^2.   \label{eqn:norm_corr}
\end{equation}
\begin{lemma}
\label{lem:sum_of_squares_1}
For each $n\ge 1$, we have
\begin{equation}
\norm{g_n}_4^4=n^2-\epsilon_n+4\sum_{1\le u\le n/2}\bigg(\frac{\sin (\pi u^2/n)}{\sin (\pi u/n)}\bigg)^2,   \label{eqn:norm_f_n}
\end{equation}
where $\epsilon_n=2$ for $n\equiv 2\pmod 4$ and $\epsilon_n=0$ otherwise.
\end{lemma}
\begin{proof}
For $f=g_n$, elementary manipulations reveal that the numbers $c_u$ in~\eqref{eqn:f_ck} satisfy
\[
\abs{c_u}=\biggabs{\frac{\sin(\pi u^2/n)}{\sin(\pi u/n)}}
\]
for $1\le u\le n-1$. The desired result then follows from~\eqref{eqn:norm_corr} after noting that $c_0=n$ and $\abs{c_u}=\abs{c_{n-u}}$ for $1\le u\le n-1$ and $2\abs{c_{n/2}}=\epsilon_n$ for even $n$.
\end{proof}
\par
We now prove Theorem~\ref{thm:main1} by finding an asymptotic evaluation of the sum on the right hand side of~\eqref{eqn:norm_f_n}. 
\par
Let $x$ be a real number satisfying $0<x\le \pi/2$. From the inequality $x-x^3/6\le \sin x\le x$ we see that
\[
0<\frac{1}{(\sin x)^2}-\frac{1}{x^2}<1,
\]
and therefore
\[
\Biggabs{\sum_{1\le u\le n/2}\bigg(\frac{\sin (\pi u^2/n)}{\sin (\pi u/n)}\bigg)^2-\sum_{1\le u\le n/2}\bigg(\frac{\sin (\pi u^2/n)}{\pi u/n}\bigg)^2}<\frac{n}{2}.
\]
Thus, defining the function $r:\R\to\R$ by
\[
r(x)=\bigg(\frac{\sin (\pi x^2/n)}{\pi x/n}\bigg)^2,
\]
the theorem is proved by showing that
\begin{equation}
\lim_{n\to\infty}\frac{1}{n^{3/2}}\sum_{1\le u\le n/2}r(u)=\frac{1}{2\pi}.   \label{eqn:sine_replaced}
\end{equation}
It is consequence of the Euler-Maclaurin formula~\cite[Theorem~B.5]{Montgomery2007} that, for real numbers $a$ and $b$ with $a<b$, the expression
\[
\Biggabs{\sum_{a<u\le b}r(u)-\int_a^br(x)\,dx}
\]
is at most
\[
\frac{1}{2}\bigg(\abs{r(a)}+\abs{r(b)}\bigg)+\frac{1}{12}\bigg(\abs{r'(a)}+\abs{r'(b)}+\int_a^b\bigabs{r''(x)}\,dx\bigg).
\]
We take $b=n/2$ and let $a$ tend to zero. Elementary calculus shows that
\[
\abs{r(n/2)}\le \frac{4}{\pi^2},\quad \abs{r'(n/2)}\le \frac{8}{\pi}+\frac{16}{n\pi^2}, \quad\lim_{a\to 0}r(a)=\lim_{a\to 0}r'(a)=0,
\]
and $\abs{r''(x)}\le 34$ for all real $x$. Therefore
\[
\Biggabs{\sum_{1\le u\le n/2}r(u)-\int_0^{n/2}r(x)\,dx}\le \frac{2}{\pi^2}+\frac{2}{3\pi}+\frac{4}{3n\pi^2}+\frac{17n}{12},
\]
and so
\[
\lim_{n\to\infty}\frac{1}{n^{3/2}}\sum_{1\le u\le n/2}r(u)=\lim_{n\to\infty}\frac{1}{n^{3/2}}\int_0^{n/2}r(x)\,dx,
\]
provided that both limits exist. Substituting $y=\pi x^2/n$, we see that this last expression equals
\[
\lim_{n\to\infty}\frac{1}{2\pi^{3/2}}\int_0^{\pi n/4}\frac{(\sin y)^2}{y^{3/2}}\,dy=\frac{1}{2\pi^{3/2}}\int_0^\infty\frac{(\sin y)^2}{y^{3/2}}\,dy.
\]
This establishes~\eqref{eqn:sine_replaced}, and so completes the proof, since
\begin{equation}
\int_0^\infty\frac{(\sin y)^2}{y^{3/2}}\,dy
=\sqrt{\pi}   \label{eqn:integral}
\end{equation}
(see Gradshteyn and Ryzhik~\cite[3.823]{Gradshteyn2007}).
\par
For completeness, we sketch a proof of the identity~\eqref{eqn:integral}. To do so, we readily verify that
\[
\frac{\Gamma(3/2)}{y^{3/2}}=\int_0^\infty e^{-yt}\sqrt{t}\,dt\quad\text{for $y>0$},
\]
which together with $\Gamma(3/2)=\sqrt{\pi}/2$ yields
\[
\int_0^\infty\frac{(\sin y)^2}{y^{3/2}}\,dy=\frac{2}{\sqrt{\pi}}\int_0^\infty \int_0^\infty e^{-yt}\sqrt{t}\,(\sin y)^2 \, dt\,dy.
\]
Since the integrand on the right hand side is nonnegative, we can interchange the order of integration by Tonelli's theorem. The integral therefore equals
\[
\frac{2}{\sqrt{\pi}}\int_0^\infty \sqrt{t}\int_0^\infty e^{-yt}\,(\sin y)^2 \, dy\,dt=\frac{2}{\sqrt{\pi}}\int_0^\infty \frac{2\sqrt{t}}{t^3+4t}\,dt=\sqrt{\pi}.
\]
The inner integral on the left hand side is just the Laplace transform of $(\sin y)^2$, while the integral on the right hand side can be evaluated by first substituting $t=x^2$ (which makes the integrand rational) and then using standard techniques.


\section{Proof of Theorem~\ref{thm:main2}}

We begin with proving a counterpart of Lemma~\ref{lem:sum_of_squares_1} for the polynomials $h_n$.
\begin{lemma}
\label{lem:norm_g}
For each $n\ge1$, we have
\[
\norm{h_n}_4^4=n^4-\gamma_n+8n\sum_{1\le v\le n/2}\;\sum_{1\le k\le v}\;\bigg(\frac{\sin(\pi k/n)}{\sin(\pi v/n)}\bigg)^2,
\]
where
\[
\gamma_n=\begin{cases}
3n^2 & \quad\text{for even $n$}\\
2n^2-2n & \quad\text{for odd $n$}.
\end{cases}
\]
\end{lemma}
\begin{proof}
Write $\z=e^{2\pi i/n}$. Then, for $f=h_n$, the numbers $c_u$ in~\eqref{eqn:f_ck} are given by (see also Turyn~\cite{Turyn1967})
\[
c_{nu+v}=\sum_{j=0}^{n-u-1}\sum_{k=0}^{n-v-1}\z^{jk-(j+u)(k+v)} + \sum_{j=0}^{n-u-2}\sum_{k=n-v}^{n-1}\z^{jk-(j+u+1)(k+v)}
\]
for $0\le u,v<n$. Rearrange and use $\sum_{k=0}^{n-1}\z^{k(u+1)}=0$ for $n\nmid u+1$ (note that the second term is zero for $u+1=n$) to see that
\begin{equation}
\overline{c_{nu+v}}=\z^{uv}\sum_{k=0}^{n-v-1}\z^{ku}\sum_{j=0}^{n-u-1}\z^{jv} - \z^{(u+1)v}\sum_{k=0}^{n-v-1}\z^{k(u+1)}\sum_{j=0}^{n-u-2}\z^{jv}   \label{eqn:c2}
\end{equation}
for $0\le u,v<n$. Evaluation of the sums over $j$ gives, for $0\le u<n$ and $0<v<n$,
\begin{align*}
\overline{c_{nu+v}}&=\frac{1}{\z^v-1}\sum_{k=0}^{n-v-1}\big(\z^{ku}(1-\z^{uv})-\z^{k(u+1)}(1-\z^{(u+1)v})\big)\\
&=\frac{1}{\z^v-1}\sum_{k=0}^{n-v-1}\big[\z^{(k+v)u}(\z^{k+v}-1)-\z^{ku}(\z^k-1)\big].
\end{align*}
We can write this as
\[
\bigg(\sum_{k=v}^{n-1}-\sum_{k=0}^{n-v-1}\bigg)\z^{ku}\;\frac{\z^k-1}{\z^v-1},
\]
from which we see that
\begin{equation}
\sum_{u=0}^{n-1}\abs{c_{nu+v}}^2=n\bigg(\sum_{k=v}^{n-1}+\sum_{k=0}^{n-v-1}-\sum_{k=v}^{n-v-1}-\sum_{k=v}^{n-v-1}\bigg)\biggabs{\frac{\z^k-1}{\z^v-1}}^2   \label{eqn:four_sums_over_c}
\end{equation}
for $0<v<n$. For $0<v<n/2$ all of these sums are nonempty, so that after grouping them together we have, for $0<v<n/2$,
\begin{align}
\sum_{u=0}^{n-1}\abs{c_{nu+v}}^2&=n\bigg(\sum_{k=n-v}^{n-1}+\sum_{k=0}^{v-1}\bigg)\biggabs{\frac{\z^k-1}{\z^v-1}}^2   \nonumber\\
&=2n\sum_{k=0}^v\;\biggabs{\frac{\z^k-1}{\z^v-1}}^2-n   \nonumber\\
&=2n\sum_{k=1}^v\;\bigg(\frac{\sin(\pi k/n)}{\sin(\pi v/n)}\bigg)^2-n.   \label{sum_c_over_u}
\end{align}
\par
Using~\eqref{eqn:c2} we readily verify that $c_{nu}=0$ for $u\ne 0$. Therefore, since $c_0=n^2$ trivially, we have from~\eqref{eqn:norm_corr}
\begin{equation}
\norm{h_n}_4^4=n^4+2\sum_{v=1}^{n-1}\;\sum_{u=0}^{n-1}\abs{c_{nu+v}}^2.   \label{eqn:norm_corr_turyn}
\end{equation}
We also have
\begin{equation}
c_{nu+v}=-\z^v\,c_{nu+n-v}\quad\text{for $(u,v)\ne(0,0)$},   \label{eqn:corr_symmetry}
\end{equation}
which also follows from~\eqref{eqn:c2} using the identities
\begin{align*}
\sum_{k=0}^{v-1}\z^{kw}&=-\z^{wv}\sum_{k=0}^{n-v-1}\z^{kw}\\
\intertext{for integers $w$ and $v$ satisfying $n\nmid w$ and $0\le v<n$ and}
\sum_{j=0}^{n-w-1}\z^{-jv}&=\z^{(w+1)v}\sum_{j=0}^{n-w-1}\z^{jv}
\end{align*}
for integers $w$ and $v$.
\par
Now, for odd $n$, we have from~\eqref{eqn:norm_corr_turyn} and~\eqref{eqn:corr_symmetry}
\[
\norm{h_n}_4^4=n^4+4\sum_{v=1}^{(n-1)/2}\;\sum_{u=0}^{n-1}\abs{c_{nu+v}}^2
\]
and the desired result follows from~\eqref{sum_c_over_u}. Similarly, for even $n$, we have
\[
\norm{h_n}_4^4=n^4+4\sum_{v=1}^{n/2-1}\;\sum_{u=0}^{n-1}\abs{c_{nu+v}}^2+2\sum_{u=0}^{n-1}\abs{c_{nu+n/2}}^2.
\]
Using~\eqref{eqn:four_sums_over_c}, we find that
\[
2\sum_{u=0}^{n-1}\abs{c_{nu+n/2}}^2=\frac{n}{2}\sum_{k=0}^{n-1}\abs{\z^k-1}^2=n^2,
\]
and therefore, by~\eqref{sum_c_over_u},
\[
\norm{h_n}_4^4=n^4-n^2+4n+8n\sum_{v=1}^{n/2-1}\sum_{k=1}^v\;\bigg(\frac{\sin(\pi k/n)}{\sin(\pi v/n)}\bigg)^2.
\]
To obtain the desired expression in the lemma for even $n$, we extend the summation over $v$ to $n/2$ and subtract the correction term
\[
8n\sum_{k=1}^{n/2}\big(\sin(\pi k/n)\big)^2=n\sum_{k=0}^{n-1}\abs{\z^k-1}^2+4n=2n^2+4n.   \qedhere
\]
\end{proof}
\par
In order to prove Theorem~\ref{thm:main2}, we invoke Lemma~\ref{lem:norm_g} and show that
\begin{equation}
8n\sum_{1\le v\le n/2}\;\sum_{1\le k\le v}\;\bigg(\frac{\sin(\pi k/n)}{\sin(\pi v/n)}\bigg)^2=\frac{4}{\pi^2}n^3+O(n^2).   \label{eqn:claim_g_norm}
\end{equation}
To do so, we make repeated use of the following elementary bound, which is also a simple consequence of the Euler-Maclaurin formula~\cite[Theorem~B.5]{Montgomery2007}. Let $r:\R\to\R$ be a differentiable function and let $a$ and $b$ be real numbers with $a<b$. Then
\begin{equation}
\Biggabs{\sum_{a<k\le b}r(k)-\int_a^br(x)\,dx}\le \frac{1}{2}\bigg(\abs{r(a)}+\abs{r(b)}+\int_a^b\bigabs{r'(x)}\,dx\bigg).   \label{eqn:EM_simple}
\end{equation}
We first take $r(x)=(\sin(\pi x/n))^2$ and $(a,b)=(0,v)$, so that for $1\le v\le n/2$, we have
\begin{align*}
\sum_{k=1}^v(\sin(\pi k/n))^2&=\int_0^v(\sin(\pi x/n))^2dx+O(1)\\
&=\frac{n}{\pi}\int_0^{\pi v/n}(\sin y)^2\,dy+O(1)\\[1.5ex]
&=\frac{n}{2\pi}\Big(\pi v/n-\sin(\pi v/n)\cos(\pi v/n)\Big)+O(1).
\end{align*}
Letting
\[
p(y)=\frac{y-\sin y\cos y}{(\sin y)^2},
\]
we then have
\begin{equation}
\sum_{1\le v\le n/2}\;\sum_{1\le k\le v}\;\bigg(\frac{\sin(\pi k/n)}{\sin(\pi v/n)}\bigg)^2=\frac{n}{2\pi}\sum_{1\le v\le n/2}p(\pi v/n)+O(n).   \label{eqn:inner_sum_int}
\end{equation}
We now apply~\eqref{eqn:EM_simple} with $r(x)=p(\pi x/n)$ and $b=n/2$ and let $a$ tend to zero. We have
\[
p'(y)=2-\frac{2(y-\sin y\cos y)\cos y}{(\sin y)^3}
\]
from which, using $x-x^3/6\le \sin x\le x$ and $1-x^2/2\le \cos x\le 1$ together with elementary calculus, we find that
\[
-3<p'(y)\le 2 \quad\text{for $0<y\le \pi/2$}.
\]
Hence $\abs{r'(x)}<3\pi/n$ for $0<x\le n/2$. Since we also have $r(n/2)=\pi/2$ and $\lim_{a\to0} r(a)=0$, we find from~\eqref{eqn:EM_simple} that~\eqref{eqn:inner_sum_int} equals
\[
\frac{n}{2\pi}\int_0^{n/2}p(\pi x/n)dx+O(n)=\frac{n^2}{2\pi^2}\int_0^{\pi/2}p(y)dy+O(n).
\]
The desired result~\eqref{eqn:claim_g_norm} is then established by showing that
\begin{equation}
\int_0^{\pi/2}p(y)dy=1.   \label{eqn:integral_one}
\end{equation}
By differentiation we readily verify that
\[
\int \frac{y-\sin y\cos y}{(\sin y)^2}\,dy=-\frac{y}{\tan y}+C
\]
for some arbitrary constant $C$ and~\eqref{eqn:integral_one} follows by application of l'H\^opital's rule.


%
\end{document}